\documentclass[11pt]{article}

\usepackage[margin=1.1in]{geometry}
\usepackage{amsmath,amssymb,amsthm,mathtools}
\usepackage{enumitem}
\usepackage{hyperref}

\hypersetup{
    colorlinks=true,
    linkcolor=blue,
    citecolor=blue,
    urlcolor=blue
}

\newtheorem{theorem}{Theorem}[section]
\newtheorem{lemma}[theorem]{Lemma}
\newtheorem{proposition}[theorem]{Proposition}

\newtheorem{conjecture}[theorem]{Conjecture}
\theoremstyle{remark}
\newtheorem{remark}[theorem]{Remark}

\newcommand{\R}{\mathbb R}
\newcommand{\Z}{\mathbb Z}
\newcommand{\Q}{\mathbb Q}

\newcommand{\T}{\mathbb R/\mathbb Z}

\newcommand{\eps}{\varepsilon}

\title{Inhomogeneous Approximation by Sums of Roots}
\author{Samuel Korsky}
\date{May 25, 2026}

\begin{document}
\maketitle

\begin{abstract}
\noindent
Let $d\geq 2$ and $k\geq 1$ be fixed. We prove that, for every $\eps>0$ and every real $\beta$, there exist integers $1\leq b_1,\ldots,b_k\leq N$ such that
\[
  \left\|\sum_{j=1}^k b_j^{1/d}-\beta\right\|
  \ll_{d,k,\eps} N^{-k/d+\eps}.
\]
The proof combines Schmidt's Subspace Theorem with an explicit inhomogeneous transference argument. This improves Iyer's (2025) higher-root exponent $(k-d+1)/d^2$, and also the analogous $d$-ary full-basis exponent away from the cases where $k+1$ is a power of $d$, at the cost of ineffectivity. We also record a conjectural uniform exponent $k-1/d$. In the square-root case $d=2$, we give explicit integer-target constructions for $k=2,3,4$ attaining this conjectural value.
\end{abstract}

\section{Introduction}

For a real number $x$, write $\|x\|$ for the distance from $x$ to the nearest integer. Iyer \cite{Iyer2024} studied sums of square roots modulo one and proved, among other things, that for every fixed $k$ and every sufficiently large $N$ there are integers $1\leq a_1,\ldots,a_k\leq N$ such that
\[
  0<\left\|\sum_{j=1}^k \sqrt{a_j}\right\| \ll_k N^{-k/2}.
\]
For arbitrary targets $\alpha\in\R$, Iyer proved that there are integers $1\leq b_1,\ldots,b_k\leq N$ with
\[
  \left\|\sum_{j=1}^k\sqrt{b_j}-\alpha\right\|
  \ll_k N^{-\gamma_k},
  \qquad
  \gamma_k=2^{\lfloor \log_2(k+1)\rfloor-1}-\frac12\geq \frac{k-1}{4}.
\]
This exponent is $k/2$ precisely when $k+1$ is a power of two. The dyadic feature comes from using the full nonconstant basis of a multiquadratic field.

\bigskip
\noindent
Iyer's later work on $\theta$-powers extends this framework to higher roots \cite{Iyer2025}. In particular, for $d$-th roots it gives an effective inhomogeneous exponent of size at least
\[
  \frac{k-d+1}{d^2},
\]
and, more sharply, a $d$-ary full-basis exponent
\[
  \gamma_{k,d}:=\frac{d^{\lfloor\log_d(k+1)\rfloor}-1}{d}.
\]
This exponent is $k/d$ precisely when $k+1$ is a power of $d$.

\bigskip
\noindent
The purpose of this note is to record a different proof of an inhomogeneous exponent, using algebraic Diophantine approximation rather than the full radical basis.

\paragraph{Relation to prior work.}
The integer-target problem for sums of square roots has been studied from both Diophantine and computational perspectives. Steinerberger \cite{Steinerberger2024} proved that, for some absolute constant $c>0$, one can make a nonzero $k$-term square-root sum $\ll_k N^{-c k^{1/3}}$ from integers at most $N$. Iyer \cite{Iyer2024} improved the integer-target exponent to $k/2$ and proved the inhomogeneous exponent $\gamma_k$ recalled above. Iyer's later work on $\theta$-powers extends the multiquadratic construction to higher roots and related power sums \cite{Iyer2025}. Earlier work of Qian and Wang \cite{QianWang2006}, Cheng and Li \cite{ChengLi2011}, and others studies the minimum nonzero difference between two signed sums of square roots, a problem arising in the complexity of comparing square-root sums. Dubickas \cite{Dubickas2024} recently considered approximate equality for two sums of $m$th roots, including modulo-one variants. The present note concerns a different but related problem: positive sums of exactly $k$ $d$-th roots approximating an arbitrary prescribed residue class modulo one.

\begin{theorem}[Main theorem]\label{thm:main}
Fix integers $d\geq 2$ and $k\geq 1$. For every $\eps>0$ there is a constant $C_{d,k,\eps}>0$ such that, for every $N\geq 2$ and every $\beta\in\R$, there exist integers
\[
  1\leq b_1,\ldots,b_k\leq N
\]
with
\[
  \left\|\sum_{j=1}^k b_j^{1/d}-\beta\right\|
  \leq C_{d,k,\eps} N^{-k/d+\eps}.
\]
The constant obtained by this proof is ineffective.
\end{theorem}

\smallskip
\noindent
Theorem~\ref{thm:main} gives the exponent $k/d-\eps$. In particular, it improves the exponent
\[
  \frac{k-d+1}{d^2}
\]
appearing in Iyer's higher-root result, after choosing
\[
  0<\eps<
  \frac{k}{d}-\frac{k-d+1}{d^2}
  =
  \frac{(d-1)(k+1)}{d^2}.
\]
It also improves the $d$-ary full-basis exponent $\gamma_{k,d}$ whenever $k+1$ is not a power of $d$, after taking $\eps<k/d-\gamma_{k,d}$. When $k+1$ is a power of $d$, Iyer's effective construction gives the endpoint exponent $k/d$, while Theorem~\ref{thm:main} loses an arbitrary $\eps$ and is ineffective.

\bigskip
\noindent
The proof has two ingredients. First, Schmidt's Subspace Theorem gives the essentially optimal dual lower bound
\[
  \max_{1\leq i\leq k}\|h p_i^{1/d}\|
  \gg_{d,k,\eta} |h|^{-1/k-\eta}
  \qquad(h\in\Z\setminus\{0\}),
\]
where $p_1,\ldots,p_k$ are distinct primes. Second, an explicit inhomogeneous transference lemma turns this dual bound into
\[
  \left\|q_1p_1^{1/d}+\cdots+q_kp_k^{1/d}-\beta\right\|
  \ll_{d,k,\eps} Q^{-k+\eps}
\]
with $|q_i|\leq Q$. After shifting the coefficients to be positive and writing $c_i p_i^{1/d}=(p_i c_i^d)^{1/d}$, this becomes a statement about sums of $k$ $d$-th roots with radicands at most $N$.

\bigskip
\noindent
We also record the following conjectural strengthening.

\begin{conjecture}[Uniform exponent]\label{conj:uniform}
For every fixed $d\geq 2$, $k\geq 1$, and $\eps>0$,
\[
  \sup_{\beta\in\T}
  \min_{1\leq b_1,\ldots,b_k\leq N}
  \left\|\sum_{j=1}^k b_j^{1/d}-\beta\right\|
  \ll_{d,k,\eps} N^{-(k-1/d)+\eps}.
\]
\end{conjecture}

\smallskip
\noindent
For the special target $\beta=0$, exact equality is trivial if all $b_j$ are $d$-th powers. The meaningful integer-target problem is therefore to ask for nonzero closeness to an integer. Define
\[
  g_{k,d}(N):=
  \min_{\substack{1\leq b_1,\ldots,b_k\leq N\\
  \sum_j b_j^{1/d}\notin\Z}}
  \left\|\sum_{j=1}^k b_j^{1/d}\right\|.
\]
The integer-target version of the same conjectural scale is
\[
  g_{k,d}(N)\ll_{d,k,\eps}N^{-(k-1/d)+\eps},
\]
or, more sharply, $g_{k,d}(N)\ll_{d,k} N^{-(k-1/d)}$ for all sufficiently large $N$. The exponent $k-1/d$ is suggested by the sensitivity $(x^{1/d})'\asymp N^{-(d-1)/d}$ near $x\asymp N$, combined with a $(k-1)$-parameter complement heuristic. In Section~\ref{sec:examples}, for $d=2$, we give explicit constructions attaining this exponent for $k=2,3,4$.

\section{Dual Lower Bound}\label{sec:subspace}

We use the following standard form of Schmidt's Subspace Theorem.

\begin{theorem}[Schmidt Subspace Theorem]\label{thm:sst}
Let $L_0,\ldots,L_n$ be linearly independent linear forms in $n+1$ variables with algebraic coefficients. For every $\delta>0$, the integer vectors $X\in\Z^{n+1}$ satisfying
\[
  \prod_{j=0}^n |L_j(X)|\leq H(X)^{-\delta}
\]
lie in finitely many proper rational subspaces of $\Q^{n+1}$. Here $H(X)=\max_i |X_i|$.
\end{theorem}

\smallskip
\noindent
We apply this only in the following elementary consequence.

\begin{lemma}[Dual lower bound]\label{lem:dual-lower}
Fix $d\geq 2$. Let $p_1,\ldots,p_k$ be distinct primes and put $\theta_i=p_i^{1/d}$. For every $\eta>0$ there is a constant $c=c(d,k,\eta,p_1,\ldots,p_k)>0$ such that
\[
  \max_{1\leq i\leq k}\|h\theta_i\|\geq c |h|^{-1/k-\eta}
  \qquad(h\in\Z\setminus\{0\}).
\]
\end{lemma}

\begin{proof}
The numbers $1,\theta_1,\ldots,\theta_k$ are linearly independent over $\Q$. This is the standard linear independence theorem for distinct radical monomials, due to Besicovitch \cite{Besicovitch1940}.

\bigskip
\noindent
Suppose first that there are infinitely many nonzero integers $h$ for which
\[
  \max_i\|h\theta_i\|< |h|^{-1/k-\eta}.
\]
For each such $h$, choose integers $m_i$ with
\[
  |h\theta_i-m_i|=\|h\theta_i\|,
\]
and set
\[
  X=(h,m_1,\ldots,m_k)\in\Z^{k+1}.
\]
Then $H(X)\asymp |h|$ along this sequence. Consider the $k+1$ algebraic linear forms
\[
  L_0(X)=X_0,
  \qquad
  L_i(X)=\theta_iX_0-X_i\quad(1\leq i\leq k).
\]
They are linearly independent. Moreover,
\[
  \prod_{i=0}^k |L_i(X)|
  = |h|\prod_{i=1}^k |h\theta_i-m_i|
  \leq |h|\left(\max_i\|h\theta_i\|\right)^k
  < |h|^{-k\eta}.
\]
Since $H(X)\asymp |h|$, the last quantity is $\ll H(X)^{-\delta}$ for, say, $\delta=k\eta/2$ and all sufficiently large $|h|$. By Theorem~\ref{thm:sst}, the corresponding integer points $X$ lie in finitely many proper rational subspaces. Hence infinitely many of them lie in one proper rational hyperplane after passing to a subsequence; that is, there are rational numbers $a_0,\ldots,a_k$, not all zero, with
\[
  a_0h+a_1m_1+\cdots+a_km_k=0
\]
for infinitely many points in the sequence. Dividing by $h$ and using $m_i/h\to\theta_i$, we obtain
\[
  a_0+a_1\theta_1+\cdots+a_k\theta_k=0,
\]
contradicting the linear independence of $1,\theta_1,\ldots,\theta_k$.

\bigskip
\noindent
Therefore all sufficiently large $|h|$ satisfy
\[
  \max_i\|h\theta_i\|\geq |h|^{-1/k-\eta}.
\]
Absorbing the finitely many remaining nonzero $h$ into the constant gives the stated bound.
\end{proof}

\section{Inhomogeneous Transference}\label{sec:transference}

The next lemma is the special transference result needed for Theorem~\ref{thm:main}. General homogeneous/inhomogeneous transference inequalities are developed in Bugeaud--Laurent \cite{BugeaudLaurent2005}; for the present one-form situation we give a direct proof using the following covering transference estimate.

\begin{lemma}[Covering transference]\label{lem:covering-transference}
Let $\Lambda\subset\R^n$ be a full-rank lattice, let $B\subset\R^n$ be a symmetric convex body with nonempty interior, let $B^\circ$ be its polar body, and let $\Lambda^*$ be the dual lattice. Then there is a constant $A_n>0$, depending only on $n$, such that
\[
  \mu(B,\Lambda)\lambda_1(B^\circ,\Lambda^*)\leq A_n.
  \tag{3.1}\label{eq:mahler}
\]
Here $\mu(B,\Lambda)$ denotes the covering radius of $B$ with respect to $\Lambda$, and $\lambda_1(B^\circ,\Lambda^*)$ is the first minimum of $B^\circ$ with respect to $\Lambda^*$. This is a standard form of Mahler's transference principle; see, for example, the transference inequalities of Banaszczyk \cite{Banaszczyk1993}, and also Cassels \cite[Chapter VIII]{Cassels1959}.
\end{lemma}

\begin{lemma}[Inhomogeneous transference for one linear form]\label{lem:transference}
Let $\theta=(\theta_1,\ldots,\theta_k)\in\R^k$. Suppose that there are constants $c>0$ and $\sigma>0$ such that
\[
  \max_{1\leq i\leq k}\|h\theta_i\|
  \geq c |h|^{-\sigma}
  \qquad(h\in\Z\setminus\{0\}).
  \tag{3.2}\label{eq:dual-assumption}
\]
Then there is a constant $C=C(k,c,\sigma)>0$ such that, for every $Q\geq 1$ and every $\beta\in\R$, there exists $q=(q_1,\ldots,q_k)\in\Z^k$ with
\[
  |q|_\infty\leq Q
\]
and
\[
  \|q\cdot\theta-\beta\|\leq C Q^{-1/\sigma}.
\]
\end{lemma}

\begin{proof}
Set $n=k+1$ and consider the unimodular lattice
\[
  \Lambda=
  \{(q_1,\ldots,q_k,q\cdot\theta-p):q\in\Z^k,\,p\in\Z\}
  \subset\R^{k+1}.
\]
For parameters $Q\geq 1$ and $\delta>0$, define the box
\[
  B=B(Q,\delta)=[-Q,Q]^k\times[-\delta,\delta].
\]
For a real target $\beta$, let
\[
  z_\beta=(0,\ldots,0,\beta)\in\R^{k+1}.
\]
If $z_\beta\in\Lambda+B$, then there are $q\in\Z^k$ and $p\in\Z$ such that
\[
  |q|_\infty\leq Q,
  \qquad
  |q\cdot\theta-p-\beta|\leq\delta,
\]
and hence $\|q\cdot\theta-\beta\|\leq\delta$.

\bigskip
\noindent
It remains to show that this covering holds when $\delta$ is a suitable multiple of $Q^{-1/\sigma}$. Suppose, to the contrary, that $z_\beta\notin\Lambda+B$. Then $\Lambda+B$ does not cover $\R^{k+1}$, so $\mu(B,\Lambda)>1$. By Lemma~\ref{lem:covering-transference}, $\lambda_1(B^\circ,\Lambda^*)\leq A_{k+1}$. The polar of $B$ is
\[
  B^\circ=
  \left\{(y_1,\ldots,y_k,t):
  Q\sum_{i=1}^k |y_i|+\delta |t|\leq 1
  \right\}.
\]
Therefore there is a nonzero vector $(y,t)\in\Lambda^*$ such that
\[
  Q\sum_{i=1}^k |y_i|+\delta |t|\leq A_{k+1}.
  \tag{3.3}\label{eq:dual-vector-bound}
\]
We now compute the dual lattice. A vector $(y,t)\in\R^k\times\R$ belongs to $\Lambda^*$ exactly when
\[
  y\cdot q+t(q\cdot\theta-p)\in\Z
\]
for every $q\in\Z^k$ and $p\in\Z$. This is equivalent to $t\in\Z$ and $y_i+t\theta_i\in\Z$ for each $i$. Thus
\[
  \Lambda^*=
  \{(m_1-h\theta_1,\ldots,m_k-h\theta_k,h):m\in\Z^k,\ h\in\Z\}.
\]
Consequently \eqref{eq:dual-vector-bound} gives integers $m_1,\ldots,m_k,h$, not all zero, with
\[
  Q\sum_{i=1}^k |m_i-h\theta_i|+\delta |h|\leq A_{k+1}.
  \tag{3.4}\label{eq:h-bound}
\]
After increasing constants we may assume $Q>A_{k+1}$. Then $h\ne0$, since if $h=0$ some $m_i\ne0$ and the left side of \eqref{eq:h-bound} is at least $Q$. Hence
\[
  |h|\leq A_{k+1}\delta^{-1}
  \tag{3.5}\label{eq:h-size}
\]
and
\[
  \max_i\|h\theta_i\|
  \leq \sum_{i=1}^k |m_i-h\theta_i|
  \leq A_{k+1}Q^{-1}.
  \tag{3.6}\label{eq:dual-small}
\]
By the hypothesis \eqref{eq:dual-assumption} and \eqref{eq:h-size},
\[
  \max_i\|h\theta_i\|
  \geq c |h|^{-\sigma}
  \geq c A_{k+1}^{-\sigma}\delta^\sigma.
  \tag{3.7}\label{eq:dual-large}
\]
Combining \eqref{eq:dual-small} and \eqref{eq:dual-large}, failure of the desired approximation implies
\[
  c A_{k+1}^{-\sigma}\delta^\sigma\leq A_{k+1}Q^{-1}.
\]
Equivalently,
\[
  \delta\leq A_{k+1}^{1+1/\sigma}c^{-1/\sigma}Q^{-1/\sigma}.
\]
Thus, if we choose
\[
  \delta=2A_{k+1}^{1+1/\sigma}c^{-1/\sigma}Q^{-1/\sigma},
\]
the failure alternative is impossible. This proves the lemma for $Q>A_{k+1}$, and the remaining bounded range of $Q$ is absorbed by increasing the constant.
\end{proof}

\section{Proof of the Main Theorem}\label{sec:main-proof}

\begin{proof}[Proof of Theorem~\ref{thm:main}]
Choose the first $k$ primes $p_1,\ldots,p_k$ and put $\theta_i=p_i^{1/d}$. Thus all constants depending on these primes depend only on $d$ and $k$. Fix $\eps>0$. Choose $\eta>0$ so small that
\[
  \frac{1}{1/k+\eta}\geq k-d\eps.
\]
By Lemma~\ref{lem:dual-lower}, the vector $\theta=(\theta_1,\ldots,\theta_k)$ satisfies
\[
  \max_i\|h\theta_i\|
  \geq c |h|^{-\sigma},
  \qquad
  \sigma=\frac1k+\eta,
\]
for some $c>0$ and all $h\ne0$. Applying Lemma~\ref{lem:transference}, we get that for every $Q\geq 1$ and every $\xi\in\R$ there is $q\in\Z^k$ with $|q|_\infty\leq Q$ and
\[
  \|q\cdot\theta-\xi\|
  \ll_{d,k,\eps} Q^{-1/\sigma}
  \leq Q^{-k+d\eps}.
  \tag{4.1}\label{eq:inhom-q}
\]

\smallskip
\noindent
Let $P=\max_i p_i$. For $N$ large, set
\[
  Q=\left\lfloor \frac{N^{1/d}}{4P^{1/d}}\right\rfloor,
  \qquad
  T=Q+1.
\]
Apply \eqref{eq:inhom-q} to the target
\[
  \xi=\beta-T(\theta_1+\cdots+\theta_k).
\]
Then there are integers $q_i$ with $|q_i|\leq Q$ such that
\[
  \left\|\sum_{i=1}^k q_i\theta_i-\beta+T\sum_{i=1}^k\theta_i\right\|
  \ll_{d,k,\eps} Q^{-k+d\eps}.
\]
Put
\[
  c_i=T+q_i.
\]
Then $1\leq c_i\leq 2Q+1$. Define
\[
  b_i=p_i c_i^d.
\]
For $N$ sufficiently large, $2Q+1\leq 3Q$, and hence
\[
  b_i\leq P(3Q)^d\leq N.
\]
Moreover,
\[
  \sum_{i=1}^k b_i^{1/d}
  =\sum_{i=1}^k c_i p_i^{1/d}
  =T\sum_{i=1}^k\theta_i+\sum_{i=1}^k q_i\theta_i.
\]
Therefore
\[
  \left\|\sum_{i=1}^k b_i^{1/d}-\beta\right\|
  \ll_{d,k,\eps} Q^{-k+d\eps}
  \ll_{d,k,\eps} N^{-k/d+\eps}.
\]
After increasing the implicit constant to handle the finitely many small values of $N$, the theorem follows.
\end{proof}

\begin{remark}[Effectivity]
The construction of the radicands from the vector $q$ is explicit once $q$ is known. The ineffective part is the lower bound in Lemma~\ref{lem:dual-lower}, inherited from the Subspace Theorem. This is the price paid for replacing Iyer's effective radical-basis norm argument by a general algebraic Diophantine approximation theorem.
\end{remark}

\section{The Conjectural \texorpdfstring{$k-1/d$}{k-1/d} Exponent}\label{sec:examples}

The proof of Theorem~\ref{thm:main} only uses the sparse family
\[
  b_i=p_i c_i^d,
  \qquad c_i\ll N^{1/d},
\]
so each $d$-th root contributes essentially one integer parameter. The conjectural exponent $k-1/d$ would require exploiting more of the local freedom in general radicands. For the integer target, this can be studied by Taylor cancellation.

\bigskip
\noindent
A natural Taylor-cancellation form of the integer-target conjecture is the following.

\begin{conjecture}[Taylor-cancellation form]\label{conj:taylor}
For every $d\geq 2$ and $k\geq 1$, there exist positive integers $A_i$, integers $u_i,v_i$, and a nonzero real constant $\lambda$ such that, for infinitely many integers $M$,
\[
  \sum_{i=1}^k A_i\left((M+u_i)^d+v_i\right)^{1/d}
  =L(M)+\lambda M^{-(dk-1)}+O_{d,k}(M^{-dk}),
\]
where $L(M)$ is an integer-valued linear polynomial in $M$.
Equivalently, after absorbing the weights $A_i$ into the radicands, one obtains
\[
  g_{k,d}(N)\ll_{d,k} N^{-(k-1/d)}
\]
for all sufficiently large $N$.
\end{conjecture}

\smallskip
\noindent
For $d=2$, this predicts the exponent $k-1/2$. We now give explicit constructions for $k=2,3,4$ showing this exponent for nonzero approximation to an integer.

\bigskip
\noindent
Let
\[
  C_t(M):=\sqrt{M^2+t}
\]
and
\[
  P(M):=\sqrt{(M+1)^2+1}+\sqrt{(M-1)^2+1}.
\]
As $M\to\infty$,
\[
  C_t(M)=M+\frac{t}{2M}-\frac{t^2}{8M^3}
  +\frac{t^3}{16M^5}-\frac{5t^4}{128M^7}+O_t(M^{-9}),
  \tag{5.1}\label{eq:C-expansion}
\]
and
\[
  P(M)=2M+\frac1M+\frac{3}{4M^3}-\frac{3}{8M^5}
  -\frac{61}{64M^7}+O(M^{-9}).
  \tag{5.2}\label{eq:P-expansion}
\]

\smallskip
\noindent
Both follow directly from the binomial expansion.

\begin{proposition}[Explicit square-root integer-target examples]\label{prop:examples}
For $k=2,3,4$, there are infinitely many $k$-tuples of positive integers $b_1,\ldots,b_k\ll M^2$ such that
\[
  0<\left\|\sum_{j=1}^k\sqrt{b_j}\right\|
  \ll_k M^{-(2k-1)}.
\]
Consequently, for every sufficiently large $N$,
\[
  g_{k,2}(N)\ll_k N^{-(k-1/2)}
  \qquad(k=2,3,4).
\]
\end{proposition}

\begin{proof}
For $k=2$, take
\[
  S_2(M)=\sqrt{M^2-1}+\sqrt{M^2+1}.
\]
By \eqref{eq:C-expansion},
\[
  S_2(M)=2M-\frac{1}{4M^3}+O(M^{-7}).
\]
Thus $0<\|S_2(M)\|\ll M^{-3}$ for all sufficiently large integers $M$.

\bigskip
\noindent
For $k=3$, take
\[
\begin{aligned}
  S_3(M)
  &:=3P(M)+2C_{-3}(M)\\
  &=\sqrt{9((M+1)^2+1)}
    +\sqrt{9((M-1)^2+1)}
    +\sqrt{4(M^2-3)}.
\end{aligned}
\]
Using \eqref{eq:C-expansion} and \eqref{eq:P-expansion},
\[
  S_3(M)=8M-\frac{9}{2M^5}+O(M^{-7}).
\]
Hence $0<\|S_3(M)\|\ll M^{-5}$.

\bigskip
\noindent
For $k=4$, take
\[
\begin{aligned}
  S_4(M)
  &:=99P(M)+108C_{-2}(M)+2C_9(M)\\
  &=\sqrt{99^2((M+1)^2+1)}
    +\sqrt{99^2((M-1)^2+1)}\\
  &\hspace{0.7cm}
    +\sqrt{108^2(M^2-2)}
    +\sqrt{2^2(M^2+9)}.
\end{aligned}
\]
Again using \eqref{eq:C-expansion} and \eqref{eq:P-expansion}, the coefficients of $M^{-1}$, $M^{-3}$, and $M^{-5}$ cancel:
\[
\begin{aligned}
  M^{-1}:&\quad 99-108+9=0,\\
  M^{-3}:&\quad 99\cdot\frac34+108\cdot\left(-\frac12\right)
  +2\cdot\left(-\frac{81}{8}\right)=0,\\
  M^{-5}:&\quad 99\cdot\left(-\frac38\right)+108\cdot\left(-\frac12\right)
  +2\cdot\frac{729}{16}=0.
\end{aligned}
\]
The first nonzero term is
\[
  99\left(-\frac{61}{64}\right)
  +108\left(-\frac58\right)
  +2\left(-\frac{32805}{128}\right)
  =-\frac{10791}{16}.
\]
Therefore
\[
  S_4(M)=308M-\frac{10791}{16M^7}+O(M^{-9}),
\]
and $0<\|S_4(M)\|\ll M^{-7}$.

\bigskip
\noindent
In all three cases the radicands are positive integers for all sufficiently large integer $M$, and their maximum is $\ll M^2$. Given any sufficiently large radicand bound $N$, choose $M=\lfloor c_k N^{1/2}\rfloor$ with $c_k>0$ small enough that all displayed radicands are at most $N$. This gives the claimed exponent $N^{-(k-1/2)}$ for every sufficiently large $N$.
\end{proof}

\smallskip
\noindent
For general $d$, the same finite cancellation problem is obtained by expanding
\[
  \left((M+u)^d+v\right)^{1/d}
  =M+u+\sum_{n\geq d-1} B_n(u,v)M^{-n}.
\]
Conjecture~\ref{conj:taylor} asks for positive rational weights $c_i$ and rational parameters $u_i,v_i$ such that
\[
  \sum_{i=1}^k c_i B_n(u_i,v_i)=0
  \qquad(d-1\leq n\leq dk-2),
\]
while the coefficient at $n=dk-1$ is nonzero. The examples in Proposition~\ref{prop:examples} solve this finite cancellation problem for $d=2$ and $k=2,3,4$.

\end{document}